\documentclass{birkjour}
\usepackage{mathrsfs}
\usepackage[colorlinks, citecolor=blue,linkcolor=blue]{hyperref}
\usepackage{amssymb,amsmath,latexsym,mathrsfs}
\usepackage{amsfonts}
 \newtheorem{thm}{Theorem}[section]
 \newtheorem{cor}[thm]{Corollary}
 \newtheorem{lem}[thm]{Lemma}
 
 \theoremstyle{definition}
 \newtheorem{defn}[thm]{Definition}
 \newtheorem{rem}[thm]{Remark}
 
 \numberwithin{equation}{section}

\begin{document}

\title[Growth and distortion results]
 {Growth and Distortion Results for a Class of Biholomorphic Mapping and Extremal Problem with Parametric Representation in $\mathbb{C}^n$}

\author[Z. Tu]{Zhenhan Tu}
\address{School of Mathematics and Statistics, Wuhan University, Wuhan, Hubei $430072$, People¡¯s Republic of China}
\email{zhhtu.math@whu.edu.cn}

\author[L. Xiong]{Liangpeng Xiong$^*$}
\address{School of Mathematics and Statistics, Wuhan University, Wuhan, Hubei $430072$, People¡¯s Republic of China}
\email{lpxiong2016@whu.edu.cn\\
\\$^*$Corresponding author}



\date{January 25, 2018}

\begin{abstract}
Let $\widehat{\mathcal {S}}_g^{\alpha, \beta}(\mathbb{B}^n)$ be a subclass of normalized biholomorphic mappings defined on the unit ball in $\mathbb{C}^n,$
which is closely related to the starlike mappings. Firstly, we obtain the growth theorem for $\widehat{\mathcal {S}}_g^{\alpha, \beta}(\mathbb{B}^n)$.
Secondly, we apply the growth theorem and a new type of the boundary Schwarz
lemma to establish the distortion theorems of the Fr\'{e}chet-derivative type
and the Jacobi-determinant type for this subclass, and the
distortion theorems with $g$-starlike mapping (resp. starlike mapping) are partly established also.
At last, we study the Kirwan and Pell type results for the compact set of mappings which have $g$-parametric representation associated with a
modified Roper-Suffridge extension operator, which extend some earlier related results.
\end{abstract}

\maketitle
\section{Introduction}\label{sec1}
In geometric theory of one complex variable, the following growth theorem and distortion theorem for biholomorphic
functions are well known.\\\\
{\bf Theorem A} (Duren \cite{dpl})\label{thA}\quad Let $f$ be a biholomorphic function on the unit disk $\mathbb{D}=\{z\in \mathbb{C}: |z|<1\},$ and
$f(0)=f'(0)-1=0$. Then
\begin{equation*}
\frac{|z|}{(1+|z|)^2}\leqslant |f(z)|\leqslant \frac{|z|}{(1-|z|)^2},\,\,\frac{1-|z|}{(1+|z|)^3}\leqslant |f'(z)|\leqslant \frac{1+|z|}{(1-|z|)^3}.
\end{equation*}\\

However, in the case of several complex variables, Cartan \cite{ch} has pointed out
that the above theorem does not hold. Therefore, in order to obtain some positive
results, it is necessary to require some
additional properties of mappings such as the convexity, the starlikeness and so on. Some results related the growth theorems for starlike mappings and the subclasses of starlike mappings on different domains can be found in many papers (see, e.g. \cite{HH2, XLT2, XLT3}). Although there are a lot of significant results which cope with the distortion theorems for convex mappings (see, e.g. \cite{bf, chhk, hkg}), there are only a few
results directly concerning the distortion theorem for starlike mappings (see \cite{lw, ll1, ll2}).
Until now, it is difficult to obtain the corresponding distortion theorems for normalized biholomorphic starlike mappings
even on the unit ball or the unit polydisc in $\mathbb{C}^n$. This paper will make progress along this lines.

To proceed further, we first introduce some notations and definitions. Denote by $\mathbb{C}^n$ the $n$-dimensional complex Hilbert space with the inner product and the norm given
$$\langle z,w\rangle=\sum\limits_{j=1}^{n}z_j\overline{{w}_j}, \,\,\|z\|=\sqrt{\langle z,z\rangle},$$
where $z, w\in \mathbb{C}^n$. Write $\mathbb{B}^n(0,r)=\{z\in \mathbb{C}^n: \|z\|=\sqrt{|z_1|^2 +...+ |z_n|^2} <r\}$ for the
open ball with center $0$ and radius $r\,\,(0<r\leqslant1)$ in $\mathbb{C}^n$.
The  boundary of $\mathbb{B}^n(0,r)$ is defined by
$\partial\mathbb{B}^n(0,r)=\{z\in \mathbb{C}^n: \|z\|=r\}.$ When $r=1$, $\mathbb{B}^n(0,1)=\mathbb{B}^n$ is the unit open ball in $\mathbb{C}^n$.
In the case of one complex
variable, $\mathbb{B}^1$ is denoted by $\mathbb{D}$.
Let $H(\mathbb{B}^n)$ be the family of all holomorphic mappings from $\mathbb{B}^n$ into
$\mathbb{C}^n$. Throughout
this paper, we write a point $z\in \mathbb{C}^n$ as a column vector in the following $n\times1$ matrix
form
$$z=\left(
      \begin{array}{c}
        z_1 \\
        z_2 \\
       \vdots\\
       z_n \\
      \end{array}
    \right),$$
and the symbol $'$ stands for the transpose of vectors or matrices. For $f\in H(\mathbb{B}^n)$, we
also write it as $f=( f_1, f_2,..., f_n)'$, where $f_j$ is a holomorphic function from $\mathbb{B}^n$
to $\mathbb{C},$ $j = 1, . . . , n.$ The derivative of $f\in H(\mathbb{B}^n)$ at a point $a\in \mathbb{B}^n$ is the complex
Jacobian matrix of $f$ given by
$$J_f(a)=\left(
           \begin{array}{c}
             \frac{\partial f_i}{\partial z_j}(a) \\
           \end{array}
         \right)_{n\times n}.
$$
If $f\in H(\mathbb{B}^n)$, we say that $f$ is normalized if
$f(0)=0$ and $J_f(0)=I_n$, where $I_n$ is the identity matrix. We say that $f\in H(\mathbb{B}^n)$ is locally biholomorphic on $\mathbb{B}^n$ if $J_f(z)$ is nonsingular at each $z\in \mathbb{B}^n$.
Let $\mathcal {L}\mathcal {S}(\mathbb{B}^n)$ be the subset of $H(\mathbb{B}^n)$ consisting of all normalized locally biholomorphic mappings on $\mathbb{B}^n.$
A holomorphic
mapping $f: \mathbb{B}^n\rightarrow \mathbb{C}^n$ is said to be biholomorphic on $\mathbb{B}^n$ if the inverse $f^{-1}$ exists and is
holomorphic on the open set $f(\mathbb{B}^n)$. Let $\mathcal {S}(\mathbb{B}^n)$ be the subset of $H(\mathbb{B}^n)$ consisting of all normalized biholomorphic
mappings on $\mathbb{B}^n.$ In the case of one complex variable, the family $\mathcal {S}(\mathbb{D})$ is denoted by
$\mathcal {S}.$

If $f, g\in H(\mathbb{B}^n),$ we say that $f$ is subordinate to $g$ (we denote by $f\prec g$) if there exists a Schwarz
mapping $v$ (i.e., $v\in H(\mathbb{B}^n)$ with $\|v(z)\|\leqslant \|z\|$, $z\in \mathbb{B}^n)$ such that $f=g\circ v.$

We next consider some subclasses of $\mathcal {S}(\mathbb{B}^n)$.
It is well-known that a locally biholomorphic mapping $f\in H(\mathbb{B}^n)$ such that
$f(0)=0$ is starlike if and only if (see, Suffridge \cite{ST})
$$\Re \langle (J_f(z))^{-1}f(z),z\rangle>0,\,\, z\in \mathbb{B}^n\backslash \{0\}.$$
The following class of mappings plays the role of the Carath\'eodory class in $\mathbb{C}^n$ (see,
e.g. Pfaltzgraff \cite{PJ1}, Suffridge \cite{ST}):
$$\mathcal {M}=\{h\in H(\mathbb{B}^n):h(0)=0, \,\,J_h(0)=I_n,\,\, \Re \langle h(z),z\rangle>0,\,\, z\in \mathbb{B}^n\backslash \{0\}\}.$$
The class $\mathcal {M}$ is important in the study of various issues in the multidimensional
geometric function theory.
\begin{defn}\label{de1.1}
Let $g\in H(\mathbb{D})$ be a univalent function such that $g(0)=1, g(\overline{\zeta})=\overline{g(\zeta)}$
for $\zeta\in \mathbb{D}$ (i.e. $g$ has real coefficients), 	$\Re g(\zeta)> 0$ on $\mathbb{D}$, and assume that $g$ satisfies
the following conditions for $r\in (0, 1):$
\begin{equation}\label{eq1}
\begin{cases}
\min\limits_{|\zeta|=r}\Re g(\zeta)=\min\{g(r),g(-r)\},&\\
 \max\limits_{|\zeta|=r}\Re g(\zeta)=\max\{g(r),g(-r)\}.&
\end{cases}
\end{equation}
The
class $\mathcal {M}_g$ was given by Graham-Hamada-Kohr \cite{GHK1} as
$$\mathcal {M}_g=\bigg\{h\in H(\mathbb{B}^n):h(0)=0,\,\, J_h(0)=I_n,\,\, \bigg\langle h(z),\frac{z}{\|z\|^2}\bigg\rangle\in g(\mathbb{D})\bigg\},$$
where $g$ satisfies the assumptions of Definition \ref{de1.1} and $z\in \mathbb{B}^n\backslash \{0\}.$ Clearly, $\mathcal {M}_g\subseteq\mathcal {M}$ and if $g(\zeta)\equiv\frac{1-\zeta}{1+\zeta}$,
then $\mathcal {M}_g\equiv\mathcal {M}$.

Let $\alpha\in [0,1),$ $\beta\in(-\frac{\pi}{2}, \frac{\pi}{2})$. We define $\widetilde{\mathcal {M}}^{\alpha,\beta}_g$ to be the class of mappings given by
\begin{align*}
\widetilde{\mathcal {M}}^{\alpha,\beta}_g&=\bigg\{h\in H(\mathbb{B}^n):h(0)=0,\,\, J_h(0)=I_n, \\
&\frac{-\alpha+\sqrt{-1}\tan \beta}{1-\alpha}+\frac{1-\sqrt{-1}\tan \beta}{1-\alpha}\bigg\langle h(z),
\frac{z}{\|z\|^2}\bigg\rangle\in g(\mathbb{D}), \,\,z\in \mathbb{B}^n\backslash \{0\}\bigg\}.
\end{align*}
Note that if $\alpha=\beta=0$, then $\widetilde{\mathcal {M}}^{0,0}_g$ coincides with the $\mathcal {M}_g$.
\end{defn}
\begin{defn}\label{de1.2}
Suppose that $g$ satisfies the assumptions of Definition \ref{de1.1} and $f\in\mathcal {L}\mathcal {S}(\mathbb{B}^n),\,\, \alpha\in(0,1)$, $\beta\in (-\frac{\pi}{2},
\frac{\pi}{2})$. The mapping $f$ is said to be in
the class $\widehat{\mathcal {S}}_g^{\alpha, \beta}(\mathbb{B}^n)$ if $[J_f(z)]^{-1}f(z)\in \widetilde{\mathcal {M}}^{\alpha,\beta}_g.$
\end{defn}
\begin{rem}
(i)\,\,If $\alpha=\beta=0$ in Definition \ref{de1.2}, then
\begin{equation*}
\frac{\bar{z}'[J_f(z)]^{-1}f(z)}{\|z\|^2}\in g(\mathbb{D}), \,\,z\in \mathbb{B}^n\backslash \{0\}.
\end{equation*}
We denote the class of $g$-starlike mappings on $\mathbb{B}^n$ by $\mathcal {S}_{g}^*(\mathbb{B}^n)$ (see \cite{GHK10}).
Obviously, if $f\in\mathcal {S}_{g}^*(\mathbb{B}^n)$, then $f$ is also a normalized starlike mapping on $\mathbb{B}^n$.\\
(ii)\,\,If $\alpha=0$ in Definition \ref{de1.2}, then
\begin{equation*}
\sqrt{-1}\frac{\sin\beta}{\cos\beta}+\frac{e^{-\sqrt{-1}\beta}}{\cos\beta}\frac{\bar{z}'[J_f(z)]^{-1}f(z)}{\|z\|^2}\in g(\mathbb{D}), \,\,z\in \mathbb{B}^n\backslash \{0\}.
\end{equation*}
We denote the class of $g$-spirallike mappings of type $\beta$ on $\mathbb{B}^n$ by $\mathcal {S}_{g}^{s*}(\mathbb{B}^n,\beta)$ (see \cite{C2}).\\
(iii)\,\,If $\beta=0$ in Definition \eqref{de1.2}, then
\begin{equation*}
\frac{1}{1-\alpha}\frac{\bar{z}'[J_f(z)]^{-1}f(z)}{\|z\|^2}-\frac{\alpha}{1-\alpha}\in g(\mathbb{D}), \,\,z\in \mathbb{B}^n\backslash \{0\}.
\end{equation*}
We denote the class of $g$-almost starlike mappings of order $\alpha$ on $\mathbb{B}^n$ by $\mathcal {S}_{g}^{as*}(\mathbb{B}^n,\alpha)$ (see \cite{C2}).
\end{rem}
We now present the notions of Loewner chains and $g$-Loewner chains (see Graham-Hamada-Kohr \cite{GHK1}).
\begin{defn}
A mappings $f : \mathbb{B}^n\times[0,\infty)\rightarrow \mathbb{C}^n$ is called a Loewner chain if $f (\cdot, t)$
is biholomorphic on $\mathbb{B}^n, f (0, t) = 0, J_f(0, t) = e^t I_n$ for $t\geqslant0$, and $f (\cdot, s)\prec f (\cdot, t)$
whenever $0\leqslant s\leqslant t <\infty.$
\end{defn}
The requirement $f (\cdot, s)\prec f (\cdot, t)$ is equivalent to the condition that there is a
unique biholomorphic Schwarz mapping $v = v(z, s, t)$, called the transition mapping
associated to $f (z, t)$ such that $f (z, s) = f (v(z, s, t), t)$ for $z\in \mathbb{B}^n, t\geqslant s\geqslant 0.$

\begin{lem}\label{le1.5}
Suppose that $h(z, t) : \mathbb{B}^n\times[0,\infty)\rightarrow \mathbb{C}^n$ satisfies the following conditions:\\
{\rm (i)}\,\,$h(\cdot,t)\in \mathcal {M}$ for $t\geqslant0$.\\
{\rm (ii)}\,\,$h(z, \cdot)$ is measurable on $[0,\infty)$ for $z\in \mathbb{B}^n$.
Let $f = f(z, t) : \mathbb{B}^n\times[0,\infty)\rightarrow \mathbb{C}^n$ be a mapping such that $f(\cdot, t)\in \mathbb{B}^n,
f(0, t)=0,\,\, J_f(0, t)=e^tI_n$ for $t\geqslant0$, and $f(z, \cdot)$ is locally absolutely continuous on
$[0,\infty)$ locally uniformly with respect to $z\in \mathbb{B}^n$. Assume that
\begin{equation}\label{eq1.2}
\frac{\partial f}{\partial t}(z, t) = J_f(z, t)h(z, t), \,\,a.e. \,\,t\geqslant0, \,\,\forall z\in\mathbb{B}^n.
\end{equation}
Further, assume that there exists an increasing sequence $\{t_m\}_{m\in\mathbb{N}}$  such that $t_m > 0,
t_m\rightarrow\infty$ and
$\lim_{m\rightarrow\infty}e^{-t_m}f(z,t_m)=F(z)$
locally uniformly on $\mathbb{B}^n$. Then $f(z, t)$ is a
Loewner chain.\end{lem}

The above characterization of Loewner chains was obtained by Graham-Hamada-Kohr \cite{GHK1}
and Pfaltzgraff \cite{PJ1}.
Now, we are able to recall the notions of a $g$-Loewner chain and $g$-parametric
representation (compare with
Chiril\u{a} \cite{C1} for
$g(\zeta)\equiv \frac{1-\zeta}{1+(1-2\gamma)\zeta}, \gamma\in (0,1), |\zeta|<1$ and compare with Graham-Hamada-Kohr-Kohr \cite{GHK10} for $g(\zeta)\equiv \frac{1-\zeta}{1+\zeta}$).

\begin{defn}
A mapping $f=f(z, t): \mathbb{B}^n\times[0,\infty)\rightarrow \mathbb{C}^n$ is called a $g$-Loewner
chain if $f(z, t)$ is a Loewner chain such that $\{e^{-t}f(\cdot, t)\}_{t\geqslant0}$ is a normal family on $\mathbb{B}^n$
and the mapping $h = h(z, t)$ in the Loewner differential equation \eqref{eq1.2} (see Lemma \ref{le1.5})
satisfies the condition $h(\cdot, t)\in \mathcal {M}_g$ for a.e. $t\geqslant 0.$
\end{defn}
Let $f :\mathbb{B}^n\rightarrow \mathbb{C}^n$ be a normalized holomorphic mapping. We say
that $f$ has $g$-parametric representation if there exists a $g$-Loewner chain $f(z, t)$ such
that $f=f(\cdot, 0).$ The notion of parametric representation was considered in  Bracci \cite{BGK}, Chiril\u{a} \cite{C2}, Graham-Hamada-Kohr \cite{GHK1}, Graham-Hamada-Kohr-Kohr \cite{GHK3},  Graham-Kohr-Pfaltzgraff \cite{GHK5}.
Let $\mathcal {S}_{g}^{0}(\mathbb{B}^n)$ be the set of mappings which have $g$-parametric representation.
\begin{rem}\label{re1.7}
(i)\,\,If $g(\zeta)\equiv \frac{1-\zeta}{1+(1-2\gamma)\zeta}, \gamma\in[0,1)$, then it reduces
to the set $\mathcal {S}_{\frac{1-\zeta}{1+(1-2\gamma)\zeta}}^{0}(\mathbb{B}^n)$ of mappings which have $\frac{1-\zeta}{1+(1-2\gamma)}$-parametric representation (see \cite{C1}).\\
(ii)\,\,If $g(\zeta)\equiv \frac{1-\zeta}{1+\zeta},$ then $\mathcal {S}_{\frac{1-\zeta}{1+\zeta}}^{0}(\mathbb{B}^n)$ reduces
to the usual set $\mathcal {S}^{0}(\mathbb{B}^n)$ of mappings which have parametric representation (see, e.g., \cite{GHK1}). It is clear that $\mathcal {S}^{0}(\mathbb{B}^n)\subseteq \mathcal {S}(\mathbb{B}^n).$\\
(iii)\,\, If $g$ is a convex function and satisfies the assumptions of Definition \ref{de1.1}, then $\mathcal {S}_{g}^{0}(\mathbb{B}^n)$ is compact in the topology of
$H(\mathbb{B}^n)$\, (see \cite{GHK10}).
\end{rem}
Let $\mathcal {F}$ be a nonempty subset of $H(\mathbb{B}^n)$. A point $f\in \mathcal {F}$ is called an extreme point of $\mathcal {F}$ provided that $f=tg+(1-t)h$, where
$t\in(0,1), g, h\in \mathcal {F},$ implies $f=g=h.$ A point $g\in \mathcal {F}$ is called a support point of $\mathcal {F}$ if there exists a continuous linear
functional $L : H(\mathbb{B}^n)\rightarrow \mathbb{C}$ such that $\Re L$ is nonconstant on $\mathcal {F}$ and
$$\Re L(g)=\max \{\Re L(h): h\in \mathcal {F}\}.$$
We denote by ex$\mathcal {F}$ and supp$\mathcal {F}$ the subsets of $\mathcal {F}$ consisting of extreme points of $\mathcal {F}$
and support points of $\mathcal {F}$, respectively (see, e.g., \cite{C2}, \cite{SS}).

For $n\geqslant2,$ let $\tilde{z}=(z_2,..., z_n)\in\mathbb{C}^{n-1}$ such that $z=(z_1, \tilde{z})\in\mathbb{C}^{n}.$
The Roper-Suffridge extension operator $\Phi_n: \mathcal {L}\mathcal {S}\rightarrow\mathcal {L}\mathcal {S}(\mathbb{B}^n)$ is defined by
$$\Phi_n(f)(z)=(f(z_1), \tilde{z}\sqrt{f'(z_1)}), \,\,z=(z_1, \tilde{z})\in \mathbb{B}^n.$$
We choose the branch of the power function such that $\sqrt{f'(z_1)}|_{z_1=0}=1$.
It provides a way of extending a locally
univalent function on the unit disc $\mathbb{D}$ to a locally biholomorphic mapping on the
Euclidean unit ball $\mathbb{B}^n$ (see \cite{RK}).

A modification of the Roper-Suffridge extension operator was given by Graham-Hamada-Kohr-Suffridge (see \cite{GHK2}):
\begin{equation}\label{eq4}
\Phi_{n,\widehat{\alpha}, \widehat{\beta}}(f)(z)=\Big(f(z_1), \tilde{z}\Big(\frac{f(z_1)}{z_1}\Big)^{\widehat{\alpha}}(f'(z_1)^{\widehat{\beta}}\Big),\,\, z=(z_1, \tilde{z})\in \mathbb{B}^n,
\end{equation}
where $\widehat{\alpha}\geqslant0, \widehat{\beta}\geqslant0$ and $f$ is a locally univalent function on $\mathbb{D}$, normalized by
$f(0)=f'(0)-1=0,$ and such that $f(z_1)\neq0$ for $z_1\in \mathbb{D}\setminus \{0\}$. We choose the
branches of the power functions such that
$(\frac{f(z_1)}{z_1})^{\widehat{\alpha}}|_{z_1=0}=1$, $(f'(z_1)^{\widehat{\beta}}|_{z_1=0}=1$.

Many mathematicians have investigated kinds of operators that preserve certain geometric and analytic properties (such as starlikeness, parametric representation,
extreme points and support points) (e.g. \cite{BGK,GHK3}).

In this paper, we organize the contents as follows. In Section \ref{sec2}, we shall establish the growth theorems
for subclass $\widehat{\mathcal {S}}_g^{\alpha, \beta}(\mathbb{B}^n)$ of biholomorphic mappings.
In Section \ref{sec3}, we shall apply the growth theorems and a new type of the boundary Schwarz
lemma for holomorphic self-mappings of the unit ball $\mathbb{B}^n$ to establish the distortion theorems of the Fr\'{e}chet-derivative type
and the Jacobi-determinant type for $\widehat{\mathcal {S}}_g^{\alpha, \beta}(\mathbb{B}^n)$ with some special points, and  the
distortion theorems associated with $g$-starlike mappings and some subclasses of $g$-starlike mappings are partly established on the unit ball in $\mathbb{C}^n$ also. In Section \ref{sec4}, we consider the
extreme points and support points with the extension operator $\Phi_{n,\hat{\alpha},\hat{\beta}}$ and $g$-parametric representation.

\section{Growth theorems associated with $\widehat{\mathcal {S}}_g^{\alpha, \beta}(\mathbb{B}^n).$}\label{sec2}
In this section, we shall obtain the growth theorems. These results generalize the conclusions in Hamada-Honda-Kohr \cite{HH2}, Xu-Liu \cite{XLT2} and Zhang \cite{XLT3}, which are important for kinds of subclasses of $g$-starlike mappings. Using Corollary \ref{co2.7},  we give an example of bounded support points for
$\mathcal {S}_g^0(\mathbb{B}^2)$ (see Graham-Hamada-Kohr-Kohr \cite{GHK10}, compare with $\mathcal {M}_g$ ).
\begin{lem}[\cite{G10}]\label{le2.1}
Suppose that $z(t):[0,1]\rightarrow \mathbb{C}^n$ is differentiable at the point $s$ which belongs to $[0,1]$, and $\|z(t)\|$ is differentiable at the point $s$
with respect to $t$. Then
$$\Re\bigg\langle \frac{dz(t)}{dt},z(t)\bigg\rangle\bigg|_{t=s}=\|z(t)\|\frac{d\|z(t)\|}{dt}\bigg|_{t=s}.$$
\end{lem}

\begin{lem}[\cite{XLT2}]\label{le2.2}
Suppose that $F$ is a mapping on $\mathbb{B}^n, z\in\mathbb{B}^n\setminus \{0\}$, and $$z(t)=F^{-1}\bigg(\exp\big(-e^{-\sqrt{-1} \beta}t\big)F(z)\bigg)(0\leqslant t<+\infty).$$
Then
{\rm (i)}\,\,$\|z(t)\|$ is strictly increasing on  $[0,+\infty)$ with respect to $t$.\\
{\rm (ii)}\,\,$\frac{dz(t)}{dt}=-e^{-\sqrt{-1} \beta}[J_{F}(z(t)]^{-1}F(z(t)).$\\
{\rm (iii)}\,\,$\frac{d\|F(z(t))\|}{dt}=-\cos\beta\|F(z(t))\|$, $\lim\limits_{t\rightarrow +\infty}\frac{\|F(z(t))\|}{\|z(t)\|}=1$.
\end{lem}

\begin{lem}\label{le2.3}
Suppose that $g$ satisfies the conditions of Definition \eqref{de1.1} and $h\in \widetilde{\mathcal {M}}^{\alpha,\beta}_g$, $\alpha\in [0,1)$, $\beta\in(-\frac{\pi}{2}, \frac{\pi}{2})$. Then the following conclusions hold:\\
{\rm (i)}\,\,$\mathcal {B}_1\leqslant \Re \langle e^{-\sqrt{-1}\beta}h(z), z\rangle\leqslant
\mathcal {B}_2$ for all $z\in \mathbb{B}^n$,
where
\begin{equation}\label{eq2.1}
\mathcal {B}_1=(1-\alpha)\cos\beta\|z\|^2\Big(\min\{g(\|z\|),g(-\|z\|)\}+\frac{\alpha}{1-\alpha}\Big)
\end{equation}
 and
 \begin{equation}\label{eq2.2}
\mathcal {B}_2=(1-\alpha)\cos\beta\|z\|^2\Big(\max\{g(\|z\|),g(-\|z\|)\}+\frac{\alpha}{1-\alpha}\Big).
\end{equation}
{\rm (ii)}\,\,In particular, if $g(z)=\frac{1+A\xi}{1+B\xi}, \xi\in \mathbb{D}, -1\leqslant A<B\leqslant1$, then
 \\$\mathcal {B}_3\leqslant|\langle h(z), z\rangle|\leqslant\mathcal {B}_4$ for all $z\in \mathbb{B}^n$, where
 \begin{equation}\label{eq2.3}
\mathcal {B}_3=\bigg(\frac{1+A\|z\|}{1+B\|z\|}-\frac{1}{1-\alpha}\sqrt{\alpha^2+\tan^{2}\beta}\bigg)(1-\alpha)\cos\beta\|z\|^2
\end{equation}
and
\begin{equation}\label{eq2.4}
\mathcal {B}_4=\bigg(\frac{1-A\|z\|}{1-B\|z\|}+\frac{1}{1-\alpha}\sqrt{\alpha^2+\tan^{2}\beta}\bigg)(1-\alpha)\cos\beta\|z\|^2.
\end{equation}
\end{lem}
\begin{proof}
Fix $z\in \mathbb{B}^n\setminus\{0\}$ and let $\mathcal {P}: \mathbb{D}\rightarrow \mathbb{C}$ be given by
\begin{equation}\label{eq2.5}
\mathcal {P}(\zeta)=
\begin{cases}
\frac{-\alpha+\sqrt{-1}\tan\beta}{1-\alpha}+\frac{1-\sqrt{-1}\tan\beta}{1-\alpha}
\frac{1}{\zeta}\langle h(\zeta \frac{z}{\|z\|}), \frac{z}{\|z\|}\rangle,& \zeta\neq0\\
 1,&\zeta=0
 \end{cases}.
 \end{equation}
Then $\mathcal {P}(\zeta)$ is a holomorphic function on $\mathbb{D}$.
Since $h\in \widetilde{\mathcal {M}}^{\alpha,\beta}_g, \mathcal {P}(0)=g(0)=1,$ it follows that $\mathcal {P}\prec g,$
and from the subordination principle it follows that $\mathcal {P}(r\mathbb{D})\subseteq g(r\mathbb{D}), r\in(0,1)$, where $r\mathbb{D}=\{z\in \mathbb{C}:|z|<r\}$. By the minimum and maximum principle for harmonic functions, we have
\begin{align}
\min\{g(|\zeta|),g(-|\zeta|)&\leqslant \frac{-\alpha}{1-\alpha}+\frac{1}{(1-\alpha)\cos\beta}
\Re \bigg[e^{-\sqrt{-1}\beta}\frac{1}{\zeta}\bigg\langle h\bigg(\zeta \frac{z}{\|z\|}\bigg),\frac{z}{\|z\|}\bigg\rangle\bigg]\notag\\
&\leqslant\max\{g(|\zeta|),g(-|\zeta|)\}.\notag
\end{align}
By letting $\zeta=\|z\|$ in the above relation, we obtain the conclusion (i).

Next, we prove the conclusion (ii). It is clear that the function $g=\frac{1+A\xi}{1+B\xi}$\\$(-1\leqslant A<B\leqslant1)$ satisfies the conditions in Definition \eqref{de1.1}.
In view of the above (i) arguments, we have $\mathcal {P}(\zeta)\prec\frac{1+A\xi}{1+B\xi}, \xi\in \mathbb{D}$. Geometrically, this subordination condition means that the image of the $\{|\zeta|<1\}$ by the function $\mathcal {P}$
is in the open disk whose two diameter endpoints are $\big[\frac{1+A|\zeta|}{1+B|\zeta|}, \frac{1-A|\zeta|}{1-B|\zeta|}\big]$. Thus, we deduce that
$$\frac{1+A|\zeta|}{1+B|\zeta|}\leqslant |\mathcal {P}(\zeta)|\leqslant\frac{1-A|\zeta|}{1-B|\zeta|}.$$
Further, using \eqref{eq2.5}, we have
$$\frac{1+A|\zeta|}{1+B|\zeta|}\leqslant \bigg|\frac{-\alpha+\sqrt{-1}\tan\beta}{1-\alpha}+\frac{1-\sqrt{-1}\tan\beta}{1-\alpha}
\frac{1}{\zeta}\Big\langle h\bigg(\zeta \frac{z}{\|z\|}\bigg), \frac{z}{\|z\|}\Big\rangle\bigg|\leqslant\frac{1-A|\zeta|}{1-B|\zeta|}.$$
Elementary computations yield that
\begin{equation}\label{eq2.6}
\bigg(\frac{1+A|\zeta|}{1+B|\zeta|}-\frac{1}{1-\alpha}\sqrt{\alpha^2+\tan^{2}\beta}\bigg)(1-\alpha)\cos\beta\leqslant
\frac{1}{|\zeta|}\Big|\Big\langle h\bigg(\zeta \frac{z}{\|z\|}\bigg), \frac{z}{\|z\|}\Big\rangle\Big|
\end{equation}
and
\begin{equation}\label{eq2.7}
\bigg(\frac{1-A|\zeta|}{1-B|\zeta|}+\frac{1}{1-\alpha}\sqrt{\alpha^2+\tan^{2}\beta}\bigg)(1-\alpha)\cos\beta\geqslant
\frac{1}{|\zeta|}\Big|\Big\langle h\Big(\zeta \frac{z}{\|z\|}\Big), \frac{z}{\|z\|}\Big\rangle\Big|.
\end{equation}
By letting $\zeta=\|z\|$ in \eqref{eq2.6} and \eqref{eq2.7}, we get the conclusion (ii).
\end{proof}

\begin{thm}\label{th2.4}
Suppose that $g$ satisfies the conditions of Definition \eqref{de1.1} and $\alpha\in [0,1)$, $\beta\in(-\frac{\pi}{2}, \frac{\pi}{2})$.
If $F\in \widehat{\mathcal {S}}_g^{\alpha, \beta}(\mathbb{B}^n)$, then $\Phi_1\leqslant \|F(z)\|\leqslant\Phi_2,\,\, z\in \mathbb{B}^n,$
where $$\Phi_1=\|z\|\exp\int_{0}^{\|z\|}\bigg(\frac{1}{(1-\alpha)\big(\max\{g(y),g(-y)\}+\frac{\alpha}{1-\alpha}\big)}-1\bigg)\frac{1}{y}dy$$and
$$\Phi_2=\|z\|\exp\int_{0}^{\|z\|}\bigg(\frac{1}{(1-\alpha)\big(\min\{g(y),g(-y)\}+\frac{\alpha}{1-\alpha}\big)}-1\bigg)\frac{1}{y}dy.$$
\end{thm}
\begin{proof}
Since $F\in \widehat{\mathcal {S}}_g^{\alpha, \beta}(\mathbb{B}^n)$, by Lemma \ref{le2.3} we have
\begin{equation}\label{eq2.8}
\mathcal {B}_1\leqslant \Re \big\langle e^{-\sqrt{-1}\beta}[J_F(z)]^{-1}F(z), z\big\rangle\leqslant
\mathcal {B}_2
\end{equation}
for all $z\in \mathbb{B}^n$, where $\mathcal {B}_1$ and $\mathcal {B}_2$ are defined by \eqref{eq2.1} and \eqref{eq2.2} respectively. Fix $z\in \mathbb{B}^n\setminus\{0\}$, let
$z(t)=F^{-1}\big(\exp(-e^{-\sqrt{-1} \beta}t)F(z)\big) \; (0\leqslant t<+\infty).$  According to (i) of Lemma \ref{le2.2}, we have that $\|z(t)\|$ is strictly increasing on  $[0,+\infty)$. Hence, $\|z(t)\|$ is differentiable on $[0,+\infty)$ almost everywhere. Denote $$\mathfrak{M}_1=\max\{g(\|z(t)\|),g(-\|z(t)\|)\}+\frac{\alpha}{1-\alpha}$$ and $$\mathfrak{M}_2=\min\{g(\|z(t)\|),g(-\|z(t)\|)\}+\frac{\alpha}{1-\alpha}.$$
From Lemma \ref{le2.1}, \ref{le2.2} and \eqref{eq2.8}, we get
$$-(1-\alpha)\cos\beta\|z(t)\|\cdot\mathfrak{M}_1\leqslant{\frac{d\|z(t)\|}{dt}}\leqslant
-(1-\alpha)\|z(t)\|\cos\beta\cdot\mathfrak{M}_2.$$
According to (iii) of Lemma \ref{le2.2}, we have
\begin{align*}
\frac{1}{\|F(z(t)\|}.\frac{d\|F(z(t))\|}{dt}&(1-\alpha)\|z(t)\|\cdot\mathfrak{M}_1\leqslant
\frac{d\|z(t)\|}{dt}\\&\leqslant\frac{1}{\|F(z(t)\|}.\frac{d\|F(z(t))\|}{dt}(1-\alpha)\|z(t)\|\cdot\mathfrak{M}_2.
\end{align*}
For any $T>0$, integrating both sides of the above inequalities with respect to $t$, we obtain
\begin{align}\label{eq2.9}
\int_0^T\frac{1}{(1-\alpha)\|z(t)\|
\mathfrak{M}_2}&\frac{d\|z(t)\|}{dt}dt\leqslant\int_0^T\frac{1}{\|F(z(t)\|}.\frac{d\|F(z(t))\|}{dt}dt\notag\\
&\leqslant\int_0^T\frac{1}{(1-\alpha)\|z(t)\|\mathfrak{M}_1}\frac{d\|z(t)\|}{dt}dt.
\end{align}
Making a chang of variable in \eqref{eq2.9}, we have
\begin{align*}
&\int_{\|z(T)\|}^{\|z\|}\frac{1}{(1-\alpha)y\bigg(\max\{g(y),g(-y)\}+\frac{\alpha}{1-\alpha}\bigg)}dy\\
&\leqslant\int_{\|F(z(T))\|}^{\|F(z)\|}\frac{1}{w}.dw\leqslant\int_{\|z(T)\|}^{\|z\|}\frac{1}{(1-\alpha)y\bigg(\min\{g(y),g(-y)\}+\frac{\alpha}{1-\alpha}
\bigg)}dy.
\end{align*}
It is elementary to verify
\begin{equation}\label{eq2.10}
\mathfrak{M}_3\leqslant\log\frac{\|F(z)\|}{\|F(z(T)\|}\leqslant\mathfrak{M}_4,
\end{equation}
where
$$\mathfrak{M}_3=\int_{\|z(T)\|}^{\|z\|}\bigg(\frac{1}{(1-\alpha)\big(\max\{g(y),g(-y)\}+\frac{\alpha}{1-\alpha}\big)}-1\bigg)\frac{1}{y}dy
+\log\frac{\|z\|}{\|z(T)\|}$$ and
$$\mathfrak{M}_4=\int_{\|z(T)\|}^{\|z\|}\bigg(\frac{1}{(1-\alpha)\big(\min\{g(y),g(-y)\}+\frac{\alpha}{1-\alpha}\big)}-1\bigg)\frac{1}{y}dy
+\log\frac{\|z\|}{\|z(T)\|}.$$
By letting $T\rightarrow+\infty$ in the above inequality \eqref{eq2.10} and using Lemma \ref{le2.2}, we finish the proof.
\end{proof}

\begin{rem}
(i)\,\,In the case of $\beta=0$, Theorem \ref{th2.4} was obtained by Zhang \cite{XLT3}.\\
(ii)\,\,In the case of $\alpha=0$, Theorem \ref{th2.4} was obtained by Zhang \cite{XLT3}, and
Xu-Liu \cite{XLT2} obtained the corresponding result in the complex Banach space.\\
(iii)\,\,In the case of $\alpha=0,\,\, \beta=0$,  Theorem \ref{th2.4} implies the results related $g$-starlike mappings on $\mathbb{B}^n$.
Hamada-Honda-Kohr \cite{HH2} obtained the corresponding result in the complex Banach space.\\
\end{rem}

\begin{cor}\label{co2.6}
Let $-1\leqslant A<B\leqslant1, \alpha\in [0,1), \beta\in(-\frac{\pi}{2}, \frac{\pi}{2}), T=A-A\alpha+B\alpha$ and $g(\xi)=\frac{1+A\xi}{1+B\xi}, \xi\in \mathbb{D}$.
If $F\in \widehat{\mathcal {S}}_g^{\alpha, \beta}(\mathbb{B}^n)$, then
\begin{align*}
\|F(z)\|&\leqslant
\begin{cases}
\|z\|[1+(A-A\alpha+B\alpha)\|z\|]^{\frac{(B-A)(1-\alpha)}{A-A\alpha+B\alpha}},&T\neq0,\\
\|z\|e^{(B-A)(1-\alpha)\|z\|},&T=0.
\end{cases}
\end{align*}
and
\begin{align*}
\|F(z)\|&\geqslant
\begin{cases}
\|z\|[1+(A\alpha-A-B\alpha)\|z\|]^{\frac{(A-B)(1-\alpha)}{A\alpha-B\alpha-A}},&T\neq0,\\
\|z\|e^{(A-B)(1-\alpha)\|z\|},&T=0.
\end{cases}
\end{align*}
\end{cor}
\begin{proof}
Since $F\in \widehat{\mathcal {S}}_g^{\alpha, \beta}(\mathbb{B}^n)$ with $g(\zeta)=\frac{1+A\zeta}{1+B\zeta}$,
Theorem \ref{th2.4} implies
\begin{align*}
\|F(z)\|&\leqslant
\|z\|\exp\int_{0}^{\|z\|}\bigg(\frac{1}{(1-\alpha)(\min\{g(y),g(-y)\}+\frac{\alpha}{1-\alpha})}-1\bigg)\frac{1}{y}dy\notag\\
&=\|z\|\exp\int_{0}^{\|z\|}\frac{(B-A)(1-\alpha)}{1+(A-A\alpha+B\alpha)x}dx\notag\\
&=
\begin{cases}
\|z\|[1+(A-A\alpha+B\alpha)\|z\|]^{\frac{(B-A)(1-\alpha)}{A-A\alpha+B\alpha}},&A-A\alpha+B\alpha\neq0,\\
\|z\|e^{(B-A)(1-\alpha)\|z\|},&A-A\alpha+B\alpha=0.
\end{cases}
\end{align*}
By the similar computations, we can obtain the lower bound also.
\end{proof}
\begin{cor}\label{co2.7}
Suppose that $g$ satisfies the conditions of Definition \ref{de1.1} and $h(z)=(h_1,h_2)\in \widetilde{\mathcal {M}}^{\alpha,\beta}_g$, $\alpha\in [0,1),\,\,
\beta\in(-\frac{\pi}{2}, \frac{\pi}{2})$,
$z=(z_1,z_2)\in \mathbb{B}^2$.
 Then
\begin{equation}\label{eq2.11}
|q^1_{0,2}|=\bigg|\frac{1}{2}\frac{\partial^2h_1}{\partial z^2_2}(0)\bigg|\leqslant \frac{3\sqrt{3}}{2}\min\{a_1,a_2\},
\end{equation}
where
\begin{equation}\label{eq2.12}
a_1=\inf\limits_{r\in(0,1)}\bigg\{\frac{1-(1-\alpha)\big(\min\{g(r),g(-r)\}+\frac{\alpha}{1-\alpha}\big)}{r}\bigg\}
\end{equation}
 and
 \begin{equation}\label{eq2.13}
 a_2=\inf\limits_{r\in(0,1)}\bigg\{\frac{(1-\alpha)\big(\max\{g(\|z\|),g(-\|z\|)\}+\frac{\alpha}{1-\alpha}\big)-1}{r}\bigg\}.
 \end{equation}
\end{cor}
\begin{proof}
Using Lemma \ref{le2.3} and taking similar arguments as those in \cite{BGK}, we can obtain the corollary immediately.
\end{proof}
\begin{rem}
\,\,Suppose that $\rho, \sigma\in \mathbb{C}\setminus\{0\}$ and $h\in H(\mathbb{B}^2)$ satisfies that $h(0)=0$ and
$h(z)=(\rho z_1+q_{0,2}^1z_2^2+O(|z_1|^2,|z_1z_2|,|| z||^3), \sigma z_2+O(|| z||^2))$ for $z=(z_1,z_2)\in \mathbb{B}^2$ (see Definition 1.3 in Bracci \cite{BGK} for references).
Then the shearing $h^{[c]}$ of $h$ is given by $h^{[c]}(z)=(\rho z_1+q_{0,2}^1z_2^2, \sigma z_2),z=(z_1,z_2)\in \mathbb{B}^2.$
In fact, if $h=(h_1,h_2)\in \widetilde{\mathcal {M}}^{\alpha,\beta}_g$, then, using the \eqref{eq2.11}, we have
\begin{equation}\label{eq2.14}
\bigg|\frac{1}{\parallel z\parallel^2}\langle h^{[c]}(z), z\rangle-1\bigg|<a_0=\min\{a_1,a_2\}, z=(z_1,z_2)\in \mathbb{B}^2\setminus \{0\},
\end{equation}
where $a_1$ and $a_2$ are defined by \eqref{eq2.12} and \eqref{eq2.13} respectively. Setting $\mathbb{D}(1,a_0)=\{z:|z-1|<a_0\}$ and if $\mathbb{D}(1,a_0)\subseteq g(\mathbb{D})$,
then $h^{[c]}\in \mathcal {M}_g$ by \eqref{eq2.14}. Graham-Hamada-Kohr-Kohr \cite{GHK10} relies on the shearing process
(see \cite{BGK}) to prove that the mapping
$F(z)=\big(z_1+\frac{3\sqrt{3}}{2}a_0z_2^2,z_2\big)\in \mathcal {S}_g^0(\mathbb{B}^2)$ and $F\in$ supp $\mathcal {S}_g^0(\mathbb{B}^2)$.
This constructs an example of a bounded starlike mapping in $\mathcal {S}_g^0(\mathbb{B}^2)$ which is a support point.
\end{rem}
\section{Distortion theorems associated with $\widehat{\mathcal {S}}_g^{\alpha, \beta}(\mathbb{B}^n).$}\label{sec3}
Establishing various versions of
the boundary Schwarz lemma has attracted the attention of many mathematicians (see, e.g. Krantz \cite{ks}, Liu-Wang-Tang \cite{lw}, Tu-Zhang \cite{TZ}). In this section, we discuss the distortion theorems
for the subclass $\widehat{\mathcal {S}}_g^{\alpha, \beta}(\mathbb{B}^n)$ of normalized biholomorphic mappings on $\mathbb{B}^n$ by using
a new type of the boundary Schwarz lemma in Liu-Wang-Tang \cite{lw}.
\begin{defn}
Let $z\in\partial \mathbb{B}^n(0, \|z\|)$.
The holomorphic tangent space \\$T_{z}^{(1,0)}(\partial \mathbb{B}^n(0, \|z\|))$ to $\partial \mathbb{B}^n(0, \|z\|)$ at $z$ is defined by
$$T_{z}^{(1,0)}(\partial \mathbb{B}^n(0, \|z\|))=\{w\in \mathbb{C}^n: \overline{z}'w=0\}.$$
If $z\in \mathbb{B}^n$, then $z_0=\frac{z}{\|z\|}\in \partial \mathbb{B}^n$. Thus, it is easy to see
$$T_{z_0}^{(1,0)}(\partial \mathbb{B}^n)=T_{z}^{(1,0)}(\partial \mathbb{B}^n(0, \|z\|)).$$
\end{defn}

\begin{lem}{\rm $($Theorem 3.2 in \cite{lw}$)$}\label{le3.2}
Let $f : \mathbb{B}^n\rightarrow \mathbb{B}^n$  be a holomorphic mapping. If $f$ is holomorphic at $z_0\in\partial \mathbb{B}^n,\,\, f(z_0)=w_0\in\partial \mathbb{B}^n$ and $z_0\neq w_0$, then the following statements hold.\\
{\rm (1)}\,\,There is $\lambda\in\mathbb{R}$ such that $\overline{J_f(z_0)}'w_0=\lambda z_0$ and $\lambda>0$.\\
{\rm (2)}\,\,$\|J_f(z_0)\delta\|\leqslant \sqrt{\lambda}\|\delta\|$ for any $\delta\in T_{z_0}^{(1,0)}(\partial \mathbb{B}^n)\cap \partial \mathbb{B}^n.$\\
{\rm (3)}\,\,$|\det J_f(z_0)|\leqslant \lambda^{\frac{n+1}{2}}$.
\end{lem}

\begin{thm}\label{th3.3}
Let $-1\leqslant A<B\leqslant1,\,\, \alpha\in [0,1),\,\, \beta\in(-\frac{\pi}{2}, \frac{\pi}{2})$ and $g(\xi)=\frac{1+A\xi}{1+B\xi}, \xi\in \mathbb{D}$. If $F\in
\widehat{\mathcal {S}}_g^{\alpha, \beta}(\mathbb{B}^n)$,
then for every $z\in \mathbb{B}^n\setminus\{0\}$, there exists an unit vector $v(z)=\frac{[J_F(z)]^{-1}F(z)}{\|[J_F(z)]^{-1}F(z)\|}$
such that
$$\|J_F(z)v(z)\|\leqslant
\begin{cases}
\frac{[1+(A-A\alpha+B\alpha)\|z\|]^{\frac{(B-A)(1-\alpha)}{A-A\alpha+B\alpha}}}{\big(\frac{1+A\|z\|}{1+B\|z\|}-\frac{1}{1-\alpha}\sqrt{\alpha^2+\tan^{2}\beta}\big)
(1-\alpha)\cos\beta},&A-A\alpha+B\alpha\neq0,\\
\frac{\|z\|e^{(B-A)(1-\alpha)}}{\big(\frac{1+A\|z\|}{1+B\|z\|}-\frac{1}{1-\alpha}\sqrt{\alpha^2+\tan^{2}\beta}\big)
(1-\alpha)\cos\beta},&A-A\alpha+B\alpha=0.
\end{cases}$$
\end{thm}
\begin{proof}
Since $F\in
\widehat{\mathcal {S}}_g^{\alpha, \beta}(\mathbb{B}^n)$, then, by taking $h(z)= [J_F(z)]^{-1}F(z)$ in Lemma \ref{le2.3}, we have
\begin{align}\label{eq3.1}
\bigg(\frac{1+A\|z\|}{1+B\|z\|}&-\frac{1}{1-\alpha}\sqrt{\alpha^2+\tan^{2}\beta}\bigg)(1-\alpha)\cos\beta\|z\|^2\notag\\&\leqslant
|\langle [J_F(z)]^{-1}F(z), z\rangle|\leqslant \|z\|\|[J_F(z)]^{-1}F(z)\|.
\end{align}
Hence, \eqref{eq3.1} implies
\begin{equation}\label{eq3.2}
\|[J_F(z)]^{-1}F(z)\|\geqslant \bigg(\frac{1+A\|z\|}{1+B\|z\|}-\frac{1}{1-\alpha}\sqrt{\alpha^2+\tan^{2}\beta}\bigg)(1-\alpha)\cos\beta\|z\|.
\end{equation}
Fix $z\in \mathbb{B}^n\setminus\{0\}$ and let $v(z)=\frac{[J_F(z)]^{-1}F(z)}{\|[J_F(z)]^{-1}F(z)\|}$.
Then
\begin{equation}\label{eq3.3}
F(z)=J_F(z)[J_F(z)]^{-1}F(z)=\|[J_F(z)]^{-1}F(z)\|J_F(z)v(z).
\end{equation}
Rewrite \eqref{eq3.3} as
\begin{equation}\label{eq3.4}
\|J_F(z)v(z)\|=\frac{\|F(z)\|}{\|[J_F(z)]^{-1}F(z)\|}.
\end{equation}
By Corollary \ref{co2.6},  \eqref{eq3.2} and \eqref{eq3.4}, we get the desired results. This completes the
proof.
\end{proof}

\begin{thm}\label{th3.4}
Let $-1\leqslant A<B\leqslant1,\,\, \alpha\in [0,1),\,\, \beta\in(-\frac{\pi}{2}, \frac{\pi}{2})$ and $g(\zeta)=\frac{1+A\zeta}{1+B\zeta}, \zeta\in \mathbb{D}$.
Suppose $f\in
\widehat{\mathcal {S}}_g^{\alpha, \beta}(\mathbb{B}^n)$. We have the following estimates.\\
{\rm (1)}\,\,If $z\in \mathbb{B}^n\backslash\{0\}$ satisfies $\|f(z)\|=\max\limits_{\|\xi\|=\|z\|}\|f(\xi)\|$, then we have
\begin{equation*}
|\det J_f(z)|\leqslant
\begin{cases}
[1+(A-A\alpha+B\alpha)\|z\|]^{\frac{n(B-A)(1-\alpha)}{A-A\alpha+B\alpha}}\bigg(\frac{1}{\digamma_1}\bigg)^{\frac{n+1}{2}},&T\neq0,\\
e^{n(B-A)(1-\alpha)\|z\|}\bigg(\frac{1}{\digamma_1}\bigg)^{\frac{n+1}{2}},&T=0
\end{cases}
\end{equation*}
and
\begin{equation*}
\|J_f(z)\delta\|\leqslant
 \begin{cases}
[1+(A-A\alpha+B\alpha)\|z\|]^{\frac{(B-A)(1-\alpha)}{A-A\alpha+B\alpha}}\bigg(\frac{1}{\digamma_1}\bigg)^{\frac{1}{2}}\|\delta\|,&T\neq0,\\
e^{(B-A)(1-\alpha)\|z\|}\bigg(\frac{1}{\digamma_1}\bigg)^{\frac{1}{2}}\|\delta\|,&T=0.
\end{cases}
\end{equation*}
{\rm (2)}\,\,If $z\in\mathbb{B}^n\backslash\{0\}$ satisfies $\|f(z)\|=\min\limits_{\|\xi\|=\|z\|}\|f(\xi)\|$, then we have
\begin{align*}
&|\det J_f(z)|\geqslant
\begin{cases}
[1+(A\alpha-A-B\alpha)\|z\|]^{\frac{n(A-B)(1-\alpha)}{A\alpha-B\alpha-A}}\digamma_2^{-\frac{n+1}{2}},&T\neq0,\\
e^{n(A-B)(1-\alpha)\|z\|}\digamma_2^{-\frac{n+1}{2}},&T=0
\end{cases}
\end{align*}
and
\begin{align*}
\|J_f(z)\delta\|
\geqslant
\begin{cases}
[1+(A\alpha-A-B\alpha)\|z\|]^{\frac{(A-B)(1-\alpha)}{A\alpha-B\alpha-A}}\digamma_2^{-\frac{1}{2}}\|\delta\|,&T\neq0,\\
e^{(A-B)(1-\alpha)\|z\|}\digamma_2^{-\frac{1}{2}}\|\delta\|,&T=0.
\end{cases}
\end{align*}
In the above {\rm (1)} and {\rm (2)}: $\delta\in T_z^{(1,0)}(\partial \mathbb{B}^n(0,\|z\|)), \,\,T=A-A\alpha+B\alpha,\,\, \digamma_1=\big(\frac{1+A\|z\|}{1+B\|z\|}-\frac{1}{1-\alpha}\sqrt{\alpha^2+\tan^{2}\beta}\big)(1-\alpha)\cos\beta,\,\,
\digamma_2= \bigg(\frac{1-A\|z\|}{1-B\|z\|}+\frac{1}{1-\alpha}\sqrt{\alpha^2+\tan^{2}\beta}\bigg)$\\$(1-\alpha)\cos\beta.$
\end{thm}
\begin{proof} (1)
Without loss of generality, we may assume that $z\neq0.$ Let $M=\|f(z)\|=\max\limits_{\|\xi\|=r}\|f(\xi)\|$ and $\|z\|= r\in(0, 1)$.
Take
\begin{equation}\label{eq3.5}
g(w)=\frac{f(rw)}{M}, \,\,w\in \mathbb{B}^n.
\end{equation}
Then $g: \mathbb{B}^n\rightarrow\mathbb{B}^n, g(0)=0$ and $g$ is biholomorphic in a neighborhood of $\overline{\mathbb{B}^n}$. Let
$z_0=\frac{z}{r}$. Then $w_0=g(z_0)=\frac{f(z)}{M}.$
Thus, it follows $z_0,\,\, w_0\in \partial \mathbb{B}^n$ and
\begin{equation}\label{eq3.6}
J_g(z_0)=\frac{r}{M}J_f(rz_0)=\frac{r}{M}J_f(z).
\end{equation}
By Lemma \ref{le3.2}, there is
$\lambda\in \mathbb{R}$ such that $\overline{J_g(z_0)}'w_0=\lambda z_0$ and
\begin{equation}\label{eq3.7}
\lambda=\overline{w_0}'J_g(z_0)z_0=\frac{\overline{f(z)}'J_f(z)z}{M^2}.
\end{equation}
From the \eqref{eq3.6} and \eqref{eq3.7}, we have that $\overline{w_0}'=\lambda \overline{z_0}'[J_g(z_0)]^{-1}$ and
\begin{equation}\label{eq3.8}
\frac{\overline{f(z)}'}{M}=\frac{\lambda M}{r^2}\overline{z}'[J_f(z)]^{-1}.
\end{equation}
This means that $\overline{f(z)}'$
and
$\overline{z}'[J_f(z)]^{-1}$ have the same direction. Since $f\in
\widehat{\mathcal {S}}_g^{\alpha, \beta}(\mathbb{B}^n)$, using Lemma \ref{le2.3} and \eqref{eq3.7}, we get
\begin{align}\label{eq3.9}
\lambda&=\frac{\overline{f(z)}'J_f(z)z}{M^2}=\frac{\overline{f(z)}'}{\|f(z)\|}\frac{J_f(z)z}{\|f(z)\|}
=\frac{\overline{z}'[J_f(z)]^{-1}}{\|\overline{z}'[J_f(z)]^{-1}\|}\frac{J_f(z)z}{\|f(z)\|}
\notag\\&=\frac{\|z\|^2}{|\overline{z}'[J_f(z)]^{-1}f(z)|}
\leqslant \frac{1}{\big(\frac{1+A\|z\|}{1+B\|z\|}-\frac{1}{1-\alpha}\sqrt{\alpha^2+\tan^{2}\beta}\big)(1-\alpha)\cos\beta}.
\end{align}
Denote by $\digamma_1=\big(\frac{1+A\|z\|}{1+B\|z\|}-\frac{1}{1-\alpha}\sqrt{\alpha^2+\tan^{2}\beta}\big)(1-\alpha)\cos\beta.$
Therefore, by Lemma \ref{le3.2} and \eqref{eq3.9}, we have
\begin{equation}\label{eq3.10}
|\det J_g(z_0)|\leqslant \lambda^{\frac{n+1}{2}}\leqslant \bigg(\frac{1}{\digamma_1}\bigg)^{\frac{n+1}{2}}
\end{equation}
and
\begin{equation}\label{eq3.11}
\|\det J_g(z_0)\delta\|\leqslant \sqrt{\lambda}\|\delta\|\leqslant\bigg(\frac{1}{\digamma_1}\bigg)^{\frac{1}{2}}\|\delta\|, \,\,\forall \delta\in T_{z_0}^{(0,1)}(\partial \mathbb{B}^n).
\end{equation}
From \eqref{eq3.6}, we have
$J_g(z_0)=\frac{r}{M}J_f(z).$  Moreover, it is easy to see $T_{z_0}^{(0,1)}(\partial \mathbb{B}^n)=T_{z}^{(1,0)}(\partial \mathbb{B}^n(0, \|z\|))$.
Hence, by Corollary \ref{co2.6}, we obtain
\begin{align*}
&|\det J_f(z)|=(\frac{M}{r})^n|\det J_g(z_0)|\leqslant (\frac{\|f(z)\|}{\|z\|})^n\bigg(\frac{1}{\digamma_1}\bigg)^{\frac{n+1}{2}}&\notag\\
&\leqslant
\begin{cases}
[1+(A-A\alpha+B\alpha)\|z\|]^{\frac{n(B-A)(1-\alpha)}{A-A\alpha+B\alpha}}\bigg(\frac{1}{\digamma_1}\bigg)^{\frac{n+1}{2}},&A-A\alpha+B\alpha\neq0,\\
e^{n(B-A)(1-\alpha)\|z\|}\bigg(\frac{1}{\digamma_1}\bigg)^{\frac{n+1}{2}},&A-A\alpha+B\alpha=0
\end{cases}
\end{align*}
and for $\forall \delta\in T_{z}^{(1,0)}(\partial \mathbb{B}^n(0, \|z\|)), $
\begin{align*}
&\|J_f(z)\delta\|=\frac{M}{r} \|J_g(z_0)\delta\|\\
&\leqslant
\begin{cases}
[1+(A-A\alpha+B\alpha)\|z\|]^{\frac{(B-A)(1-\alpha)}{A-A\alpha+B\alpha}}\bigg(\frac{1}{\digamma_1}\bigg)^{\frac{1}{2}}\|\delta\|,&A-A\alpha+B\alpha\neq0,\\
e^{(B-A)(1-\alpha)\|z\|}\bigg(\frac{1}{\digamma_1}\bigg)^{\frac{1}{2}}\|\delta\|,&A-A\alpha+B\alpha=0.
\end{cases}
\end{align*}
(2)
Let $m=\|f(z)\|=\max\limits_{\|\xi\|=r}\|f(\xi)\|$ and $\|z\|= r\in(0, 1)$.
Take
\begin{equation}\label{eq3.12}
h(w)=\frac{f(rw)}{m}, \,\,w\in \mathbb{B}^n.
\end{equation}
Then $h(0)=0$ and $h$ is biholomorphic in a neighborhood of $\overline{\mathbb{B}^n}$ with $h(\mathbb{B}^n)\supset \mathbb{B}^n$. Let $z_0=\frac{z}{r}$. Then $w_0=h(z_0)=\frac{f(z)}{m}.$
Thus, it follows $z_0,\,\, w_0\in \partial \mathbb{B}^n$ and $J_h(z_0)=\frac{r}{m}J_f(rz_0)=\frac{r}{m}J_f(z).$
Furthermore, we have that $h^{-1}: \mathbb{B}^n\rightarrow\mathbb{B}^n, h^{-1}(0)=0$ and $h^{-1}$ is biholomorphic in a neighborhood of $\overline{\mathbb{B}^n}$
with $h^{-1}(w_0)=z_0$.
By the similar proof as in (1), we also conclude that $\overline{f(z)}'$
and
$\overline{z}'[J_f(z)]^{-1}$ have the same direction. Since $f\in
\widehat{\mathcal {S}}_g^{\alpha, \beta}(\mathbb{B}^n)$, using Lemma \ref{le2.3} and Lemma \ref{le3.2}, there is
$\lambda\in \mathbb{R}$ such that
\begin{align}\label{eq3.13}
\lambda&=\overline{z_0}'J_{h^{-1}}(w_0)w_0=\overline{z_0}'[J_h(z_0)]^{-1}w_0=\frac{\overline{z}'}{r}\big[\frac{r}{m}J_f(z)\big]^{-1}\frac{f(z)}{m}\notag\\&=
\frac{\overline{z}'}{\|z\|}\frac{\|f(z)\|}{\|z\|}[J_f(z)]^{-1}\frac{f(z)}{\|f(z)\|}
=\frac{\overline{z}'[J_f(z)]^{-1}f(z)}{\|z\|^2}\notag\\
&\leqslant \bigg(\frac{1-A\|z\|}{1-B\|z\|}+\frac{1}{1-\alpha}\sqrt{\alpha^2+\tan^{2}\beta}\bigg)(1-\alpha)\cos\beta.
\end{align}
Denote by $\digamma_2= \bigg(\frac{1-A\|z\|}{1-B\|z\|}+\frac{1}{1-\alpha}\sqrt{\alpha^2+\tan^{2}\beta}\bigg)(1-\alpha)\cos\beta.$ Thus,
using Lemma \ref{le3.2} and \eqref{eq3.13}, we have
\begin{equation}\label{eq3.14}
|\det J_{h^{-1}}(w_0)|=\frac{1}{|\det J_h(z_0)|}\leqslant\lambda^{\frac{n+1}{2}}\leqslant \digamma_2^{\frac{n+1}{2}}
\end{equation}
and
\begin{equation}\label{eq3.15}
\|J_{h^{-1}}(w_0)\widehat{\delta}\|\leqslant \sqrt{\lambda}\|\widehat{\delta}\|\leqslant \digamma_2^{\frac{1}{2}}\|\widehat{\delta}\|, \,\,\forall \widehat{\delta}\in T_{w_0}^{(0,1)}(\partial \mathbb{B}^n).
\end{equation}
Note
$J_h(z_0)=\frac{r}{m}J_f(z).$ Moreover, it is easy to see $$T_{z_0}^{(0,1)}(\partial \mathbb{B}^n)=T_{z}^{(1,0)}(\partial \mathbb{B}^n(0, \|z\|)).$$
Hence, by \eqref{eq3.14}, we can obtain
\begin{align*}
&\frac{1}{|\det J_f(z)|}=\big(\frac{r}{m}\big)^n\frac{1}{|\det J_h(z_0)|}=\big(\frac{\|z\|}{\|f(z)\|}\big)^n\frac{1}{|\det J_h(z_0)|}\leqslant \big(\frac{\|z\|}{\|f(z)\|}\big)^n\digamma_2^{\frac{n+1}{2}}.
\end{align*}
Combining Corollary \ref{co2.6} and the above result, we have
\begin{align*}
&|\det J_f(z)|\geqslant \big(\frac{\|f(z)\|}{\|z\|}\big)^n\digamma_2^{-\frac{n+1}{2}}\\
&\geqslant
\begin{cases}
[1+(A\alpha-A-B\alpha)\|z\|]^{\frac{n(A-B)(1-\alpha)}{A\alpha-B\alpha-A}}\digamma_2^{-\frac{n+1}{2}},&A-A\alpha+B\alpha\neq0,\\
e^{n(A-B)(1-\alpha)\|z\|}\digamma_2^{-\frac{n+1}{2}},&A-A\alpha+B\alpha=0.
\end{cases}
\end{align*}
Note that $J_f(z)T_z^{(1,0)}(\partial \mathbb{B}^n(0,\|z\|))\subset T_{w_0}^{(1,0)}(\partial \mathbb{B}^n)$ (see the proof of Theorem 3.2 in \cite{lw})
and
\begin{equation}\label{eq3.16}
J_{h^{-1}}(w_0)=\frac{\|f(z)\|}{\|z\|}[J_f(z)]^{-1}.
\end{equation}
When the $\widehat{\delta}$ is replaced by $J_f(z)\delta$ in \eqref{eq3.15}, by using \eqref{eq3.16},  we have
\begin{equation}\label{eq3.17}
\digamma_2^{\frac{1}{2}}\|J_f(z)\delta\|\geqslant\|J_{h^{-1}}(w_0)J_f(z)\delta\|=\frac{\|f(z)\|}{\|z\|}\|\delta\|,\,\,\forall \delta\in T_z^{(1,0)}(\partial \mathbb{B}^n(0,\|z\|)).
\end{equation}
In view of Corollary \ref{co2.6} and \eqref{eq3.17}, we get
\begin{align*}
&\digamma_2^{\frac{1}{2}}\|J_f(z)\delta\|\geqslant\frac{\|f(z)\|}{\|z\|}\|\delta\|\\
&\geqslant
\begin{cases}
[1+(A\alpha-A-B\alpha)\|z\|]^{\frac{(A-B)(1-\alpha)}{A\alpha-B\alpha-A}}\|\delta\|,&A-A\alpha+B\alpha\neq0,\\
e^{(A-B)(1-\alpha)\|z\|}\|\delta\|,&A-A\alpha+B\alpha=0.
\end{cases}
\end{align*}
Thus, for $\forall \delta\in T_z^{(1,0)}(\partial \mathbb{B}^n(0,\|z\|))$, we have
\begin{align*}
&\|J_f(z)\delta\|\geqslant\frac{\|f(z)\|}{\|z\|}\|\delta\|\\
&\geqslant
\begin{cases}
[1+(A\alpha-A-B\alpha)\|z\|]^{\frac{(A-B)(1-\alpha)}{A\alpha-B\alpha-A}}\digamma_2^{-\frac{1}{2}}\|\delta\|,&A-A\alpha+B\alpha\neq0,\\
e^{(A-B)(1-\alpha)\|z\|}\digamma_2^{-\frac{1}{2}}\|\delta\|,&A-A\alpha+B\alpha=0.
\end{cases}
\end{align*}
which gives the desired result.
\end{proof}

When $\alpha=\beta=0$ or $\alpha=\beta=0, A=-1, B=1-2\gamma(0<\gamma<1)$,  Theorem \ref{th3.4} implies the Corollary \ref{co3.5} and Corollary \ref{co3.6} as follows, which are related the $g$-starlike mappings and starlike mappings of order $\gamma$ ($0<\gamma<1$) on $\mathbb{B}^n$.

\begin{cor}\label{co3.5}
Let $-1\leqslant A<B\leqslant1$ and $g(\zeta)=\frac{1+A\zeta}{1+B\zeta}, \zeta\in \mathbb{D}$.
Suppose $f\in
\mathcal {S}_g^*(\mathbb{B}^n)$. We have the following estimates.\\
{\rm (1)}\,\,If $z\in \mathbb{B}^n\backslash\{0\}$ satisfies $\|f(z)\|=\max\limits_{\|\xi\|=\|z\|}\|f(\xi)\|$, then we have
\begin{equation*}
|\det J_f(z)|\leqslant
\begin{cases}
(1+A\|z\|)^{\frac{n(B-A)}{A}}\bigg(\frac{1+B\|z\|}{1+A\|z\|}\bigg)^{\frac{n+1}{2}},&A\neq0,\\
e^{n(B-A)\|z\|}\bigg(\frac{1+B\|z\|}{1+A\|z\|}\bigg)^{\frac{n+1}{2}},&A=0
\end{cases}
\end{equation*}
and
\begin{equation*}
\|J_f(z)\delta\|\leqslant
 \begin{cases}
(1+A\|z\|)^{\frac{(B-A)}{A}}\bigg(\frac{1+B\|z\|}{1+A\|z\|}\bigg)^{\frac{1}{2}}\|\delta\|,&A\neq0,\\
e^{(B-A)\|z\|}\bigg(\frac{1+B\|z\|}{1+A\|z\|}\bigg)^{\frac{1}{2}}\|\delta\|,&A=0.
\end{cases}
\end{equation*}
{\rm (2)}\,\,If $z\in\mathbb{B}^n\backslash\{0\}$ satisfies $\|f(z)\|=\min\limits_{\|\xi\|=\|z\|}\|f(\xi)\|$, then we have
\begin{align*}
&|\det J_f(z)|\geqslant
\begin{cases}
[1-A\|z\|]^{\frac{n(A-B)(1-\alpha)}{-A}}\big(\frac{1-A\|z\|}{1-B\|z\|}\big)^{-\frac{n+1}{2}},&A\neq0,\\
e^{n(A-B)\|z\|}\big(\frac{1-A\|z\|}{1-B\|z\|}\big)^{-\frac{n+1}{2}},&A=0
\end{cases}
\end{align*}
and
\begin{align*}
\|J_f(z)\delta\|
\geqslant
\begin{cases}
(1-A\|z\|)^{\frac{(A-B)}{-A}}\big(\frac{1-A\|z\|}{1-B\|z\|}\big)^{-\frac{1}{2}}\|\delta\|,&A\neq0,\\
e^{(A-B)\|z\|}\big(\frac{1-A\|z\|}{1-B\|z\|}\big)^{-\frac{1}{2}}\|\delta\|,&A=0.
\end{cases}
\end{align*}
In above {\rm (1)} and {\rm (2)}: $\delta\in T_z^{(1,0)}(\partial \mathbb{B}^n(0,\|z\|))$.
\end{cor}

\begin{cor}\label{co3.6}
Let $f(z)$ be a normalized biholomorphic  starlike mappings of order $\gamma$ {\rm ($0<\gamma<1$)} on $\mathbb{B}^n$. We have the following estimates.\\
{\rm (1)}\,\,If $z\in \mathbb{B}^n\backslash\{0\}$ satisfies $\|f(z)\|=\max\limits_{\|\xi\|=\|z\|}\|f(\xi)\|$, then we have
\begin{equation*}
|\det J_f(z)|\leqslant
(1-\|z\|)^{2(\gamma-1) n}\bigg(\frac{1-\|z\|}{1+(1-2\gamma)\|z\|}\bigg)^{-\frac{n+1}{2}},
\end{equation*}
and
\begin{equation*}
\|J_f(z)\delta\|\leqslant
(1-\|z\|)^{2(\gamma-1)}\bigg(\frac{1-\|z\|}{1+(1-2\gamma)\|z\|}\bigg)^{-\frac{1}{2}}\|\delta\|.
\end{equation*}
{\rm (2)}\,\,If $z\in\mathbb{B}^n\backslash\{0\}$ satisfies $\|f(z)\|=\min\limits_{\|\xi\|=\|z\|}\|f(\xi)\|$, then we have
\begin{align*}
&|\det J_f(z)|\geqslant
(1+\|z\|)^{-2(1-\gamma)n}\bigg(\frac{1+\|z\|}{1-(1-2\gamma)\|z\|}\bigg)^{-\frac{n+1}{2}},
\end{align*}
and
\begin{align*}
\|J_f(z)\delta\|
\geqslant
(1+\|z\|)^{-2(1-\gamma)}\bigg(\frac{1+\|z\|}{1-(1-2\gamma)\|z\|}\bigg)^{-\frac{1}{2}}\|\delta\|.
\end{align*}
In above {\rm (1)} and {\rm (2)}: $\delta\in T_z^{(1,0)}(\partial \mathbb{B}^n(0,\|z\|))$.
\end{cor}
\begin{rem}
(i)\,\,If we take $A=-1,\,\, B=1$ in Corollary \ref{co3.5}, then
$f$ becomes a normalized biholomorphic starlike mappings on $\mathbb{B}^n$. The corresponding result was partly proved by Liu-Wang-Tang \cite{lw}.\\
(ii)\,\,In particular, when $n=1,\,\, A=-1,\,\, B=1$ in Corollary \ref{co3.5}, the result coincides with the classical distortion theorem for starlike function in one complex variable (see, Theorem A).
\end{rem}

\section{Extreme and support points associated with $\Phi_{n,\widehat{\alpha},\widehat{\beta}}$ and $g$-parametric representation.}\label{sec4}
 Kirwan \cite{KK} and Pell \cite{PR} proved that if $f\in$ ex$\mathcal {S}$ (resp., $f\in$ supp$\mathcal {S}$) and
$f(z, t)$ is a Loewner chain such that $f = f (\cdot, 0)$, then $e^{-t} f (\cdot, t)\in$ ex$\mathcal {S}$ (resp., $e^{-t}f(\cdot, t)\in$ supp$\mathcal {S}$).
Graham-Hamada-Kohr-Kohr \cite{GHK4} proved
that if $f\in$ ex$\mathcal {S}^0(\mathbb{B}^n)$ (resp., $f\in$ supp${\mathcal {S}}^0(\mathbb{B}^n)$) and
$f(z, t)$ is a Loewner chain such that $f = f (\cdot, 0)$ on $\mathbb{B}^n$, then $e^{-t} f (\cdot, t)\in$ ex$\mathcal {S}^0(\mathbb{B}^n)$
for $t\geqslant0$ (resp., $e^{-t}f(\cdot, t)\in$ supp$\mathcal {S}^0(\mathbb{B}^n)$ for $0\leqslant t<t_0$).
Chiril\u{a}-Hamada-Kohr \cite{C4} proved that these results hold for the set of mappings which have $g$-parametric representation.  Furthermore,
Chiril$\breve{a}$ \cite{C2} and Graham-Kohr-Pfaltzgraff \cite{GHK5} proved that Roper-Suffridge extension operator and a general Pfaltzgraff-Suffridge extension operator $\Psi_{n,\widehat{\alpha}}$ preserve these properties with extreme and support points, where $\Psi_{n,\widehat{\alpha}}$ is defined as $\Psi_{n,\widehat{\alpha}}: \mathcal {L}\mathcal {S}_n(\mathbb{B}^n)\rightarrow
\mathcal {L}\mathcal {S}_{n+1}(\mathbb{B}^{n+1}), \widehat{\alpha}\geqslant0, z=(z_1,z_2,...,z_{n+1})\in \mathbb{B}^{n+1},$
$$\Psi_{n,\widehat{\alpha}}(f)(z)=(f(z_1,z_2,...,z_n),z_{n+1}[J_f(z_1,z_2,...,z_n)]^{\widehat{\alpha}}).$$
In fact, if $\widehat{\alpha}=\frac{1}{n+1}$, the $\Psi_{n,\widehat{\alpha}}$ coincides with the Pfaltzgraff-Suffridge extension operator $\Psi_{n}$. In 2014,
Chiril\u{a}-Hamada-Kohr \cite{C4} studied the above problems on extreme and support points with $\Psi_{n}$.

Motivated by the above works, we will prove that this result also holds in the modified Roper-Suffridge extension operator $\Phi_{n,\widehat{\alpha},\widehat{\beta}}$ with the set $\mathcal {S}^0_g(\mathbb{B}^n)$ of mappings which have $g$-parametric representation. Since
$\mathcal {S}^0_g(\mathbb{B}^n)$ is a compact (see, Remark \ref{re1.7}(iii)), and thus a closed subset of $\mathcal {S}(\mathbb{B}^n)$.
Note that the $\Phi_{n,\widehat{\alpha},\widehat{\beta}}: \mathcal {L}\mathcal {S}\rightarrow\mathcal {L}\mathcal {S}(\mathbb{B}^n)$ is different from the general Pfaltzgraff-Suffridge extension operator $\Psi_{n,\widehat{\alpha}}$.

\begin{thm}\label{th4.1}
Let $f\in \overline{\mathcal {S}^0_g(\mathbb{D})}$ and let $F=\Phi_{n,\widehat{\alpha},\widehat{\beta}}(f)$ with $\widehat{\alpha}\in [0, 1]$, $\widehat{\beta}\in [0, \frac{1}{2}]$ and
$\widehat{\alpha}+\widehat{\beta}\leqslant1$.
Let $F_{\widehat{\alpha},\widehat{\beta}}(z,t)$ be the $g$-Loewner
chain given by the below \eqref{eq4.1} such that $F=F_{\widehat{\alpha},\widehat{\beta}}(z,0).$ If $F\in$ {\rm ex}$\Phi_{n,\widehat{\alpha},\widehat{\beta}}\Big(\overline{\mathcal {S}^0_g(\mathbb{D}})\Big)$,
then $e^{-t}F_{\widehat{\alpha},\widehat{\beta}}(\cdot,t)\in$ {\rm ex}$\Phi_{n,\widehat{\alpha},\widehat{\beta}}\Big(\overline{\mathcal {S}^0_g(\mathbb{D})}\Big)$
for $t\geqslant0.$
\end{thm}
\begin{proof}Here we use the similar way to those in the proof of Chiril\u{a} \cite{C2}.
Since $f\in \overline{\mathcal {S}^0_g(\mathbb{D})}$, there exists a $g$-Loewner chain $f_t(z_1)=f(z_1,t): \mathbb{D}\times[0,+\infty)\rightarrow \mathbb{C}$ such that
$f=f(\cdot,0)$ and $\{e^{-t}f(\cdot,t)\}$ is a normal family on $\mathbb{D}.$ Set
\begin{equation}\label{eq4.1}
F_{\widehat{\alpha},\widehat{\beta}}(z,t)=\Bigg(f(z_1,t), \tilde{z}e^{(1-\widehat{\alpha}-\widehat{\beta})t}\bigg(\frac{f(z_1,t)}{z_1}\bigg)^{\widehat{\alpha}}(f'(z_1,t))^{\widehat{\beta}}\Bigg),
\end{equation}
$z=(z_1,\tilde{z})\in \mathbb{B}^n, t\geqslant0,$ where the branch of the power function is chosen such that $(\frac{f(z_1)}{z_1})^{\widehat{\alpha}}|_{z_1=0}=1$, $(f'(z_1)^{\widehat{\beta}}|_{z_1=0}=1$.
Then $F_{\widehat{\alpha},\widehat{\beta}}(z,t)$ is a $g$-Loewner chain such that
$F=F_{\widehat{\alpha},\widehat{\beta}}(\cdot,0)$  and $ \{e^{-t}F_{\widehat{\alpha},\widehat{\beta}}(\cdot,t)\}$ is a normal family on $\mathbb{B}^n$ (see Theorem 2.1 in Chiril\u{a} \cite{C1}).
Moreover, it is clear that $e^{-t}F_{\widehat{\alpha},\widehat{\beta}}(\cdot,t)=\Phi_{n,\widehat{\alpha},\widehat{\beta}}(e^{-t}f(\cdot,t))\in \Phi_{n,\widehat{\alpha},\widehat{\beta}}\Big(\overline{\mathcal {S}^0_g(\mathbb{D})}\Big)$
for $t\geqslant0$. Let $v_{s,t}(z_1)=v(z_1,s,t)$ be the transition mapping associated
with $f(z_1,t)$ and let $V(z,s,t)$ be the transition mapping associated
with $F_{\widehat{\alpha},\widehat{\beta}}(z,t)$. By Theorem 3.1 in \cite{C1}, we get
\begin{equation}\label{eq4.2}
V(z,s,t)=\Big(v_{s,t}(z_1),\tilde{z}e^{(s-t)(1-\widehat{\alpha}-\widehat{\beta})}\Big(\frac{v(z_1,s,t)}{z_1}\Big)^{\widehat{\alpha}}(v'(z_1,s,t))^{\widehat{\beta}}\Big),
\end{equation}
where $z=(z_1,\tilde{z})\in \mathbb{B}^n, t\geqslant s\geqslant0,$ the branch of the power function is chosen such that
$(\frac{v(z_1,s,t)}{z_1})^{\widehat{\alpha}}|_{z_1=0}=e^{(s-t)\widehat{\alpha}}$ and $(v'(z_1,s,t))^{\widehat{\alpha}}|_{z_1=0}=e^{(s-t)\widehat{\beta}}.$
\\
Fix $t\geqslant0$. Let
\begin{equation}\label{eq4.3}
e^{-t}F_{\widehat{\alpha},\widehat{\beta}}(z,t)=\lambda M(z)+(1-\lambda)G(z), z\in (z_1,\tilde{z})\in \mathbb{B}^n,
\end{equation}
where $\lambda\in (0,1)$ and $M$, $G\in \Phi_{n,\widehat{\alpha},\widehat{\beta}}\Big(\overline{\mathcal {S}^0_{g}(\mathbb{D})}\Big)$.
Taking $V(z,t)=V(z,0,t)$ for $z\in \mathbb{B}^n$. By \eqref{eq4.3}, we obtain
\begin{align}\label{eq4.4}
F(z)&=F_{\widehat{\alpha},\widehat{\beta}}(z,0)=F_{\widehat{\alpha},\widehat{\beta}}(V(z,t),t)\notag\\&=\lambda e^tM(V(z,t))+(1-\lambda)e^tG(V(z,t))
\end{align}
for $z\in (z_1,z')\in \mathbb{B}^n.$ Let $m$, $g\in \overline{\mathcal {S}^0_g(\mathbb{D})}$ with $M=\Phi_{n,\widehat{\alpha},\widehat{\beta}}(m), G=\Phi_{n,\widehat{\alpha},\widehat{\beta}}(g).$ Also let $v_t(z_1)=v(z_1,t)=v(z_1,0,t).$
It is obvious that
$$e^tM(V(z,t))=\Phi_{n,\widehat{\alpha},\widehat{\beta}}(e^t (m\circ v_t))(z)$$ and
$$e^tG(V(z,t))=\Phi_{n,\widehat{\alpha},\widehat{\beta}}(e^t (g\circ v_t))(z).$$
By \eqref{eq4.2}, we have
\begin{align}\label{eq4.5}
&e^tM(V(z,t))=e^t\Phi_{n,\widehat{\alpha},\widehat{\beta}}(m)(V(z,t))\notag\\&=\Bigg(e^tm(v_t(z_1)),\tilde{z}e^{t(\widehat{\alpha}+\widehat{\beta})}
\bigg(\frac{v_t(z_1)}{z_1}\bigg)^{\widehat{\alpha}}(v'_t(z_1))^{\widehat{\beta}}
\bigg(\frac{m(v_t(z_1))}{v_t(z_1)}\bigg)^{\widehat{\alpha}}(m'(v_t(z_1))^{\widehat{\beta}}\Bigg)\notag\\&=\Bigg(e^tm(v_t(z_1)),\tilde{z}e^{t(\widehat{\alpha}+\widehat{\beta})}
\Bigg(\frac{m(v_t(z_1))}{z_1}\Bigg)^{\widehat{\alpha}}(v'_t(z_1))^{\widehat{\beta}}(m'(v_t(z_1))^{\widehat{\beta}}\Bigg)
\notag\\&=\Bigg(e^t (m\circ v_t)(z_1),\tilde{z}\Bigg(\frac{e^t (m\circ v_t)(z_1)}{z_1}\Bigg)^{\widehat{\alpha}}(e^t (m\circ v_t)'(z_1))^{\widehat{\beta}}\Bigg)
\notag\\&=\Phi_{n,\widehat{\alpha},\widehat{\beta}}(e^t (m\circ v_t))(z).
\end{align}
Similarly, $e^tG(V(z,t))=\Phi_{n,\widehat{\alpha},\widehat{\beta}}(e^t (g\circ v_t))(z)$ for $z\in \mathbb{B}^n$.
Further, since $m\in \overline{\mathcal {S}^0_g(\mathbb{D})},$ it follows that the composition $e^t (m\circ v_t))(z_1)$ is a function in $\overline{\mathcal {S}^0_g(\mathbb{D})}$. Hence
$e^tM(V(z,t))\in \Phi_{n,\widehat{\alpha},\widehat{\beta}}\Big(\overline{\mathcal {S}^0_g(\mathbb{D})}\Big)$. Similarly, $e^tG(V(z,t))\in \Phi_{n,\widehat{\alpha},\widehat{\beta}}\Big(\overline{\mathcal {S}^0_g(\mathbb{D})}\Big)$. From \eqref{eq4.4} and \eqref{eq4.5}, we have
$$F(z)=\lambda \Phi_{n,\widehat{\alpha},\widehat{\beta}}(e^t m\circ v_t)(z)+(1-\lambda)\Phi_{n,\widehat{\alpha},\widehat{\beta}}(e^t g\circ v_t)(z),\; z\in (z_1,z')\in \mathbb{B}^n.$$
Since $F\in$ ex $\Phi_{n,\widehat{\alpha},\widehat{\beta}}\Big(\overline{\mathcal {S}^0_g(\mathbb{D})}\Big)$, we have $$\Phi_{n,\widehat{\alpha},\widehat{\beta}}(e^t m\circ v_t)(z)\equiv\Phi_{n,\widehat{\alpha},\widehat{\beta}}(e^t g\circ v_t)(z)$$ for $z\in \mathbb{B}^n$. Finally, applying the identity theorem for holomorphic mappings, we
get $\Phi_{n,\widehat{\alpha},\widehat{\beta}}(m)\equiv\Phi_{n,\widehat{\alpha},\widehat{\beta}}(g)$, i.e., $M\equiv G.$ This completes the proof.
\end{proof}

We next consider the analog of a result of Pell \cite{PR} and Chiril\u{a} \cite{C2} concerning support
points and Loewner chains associated with the Roper-Suffridge extension operator $\Phi_{n,\widehat{\alpha},\widehat{\beta}}$.
Here we use the similar way to those in the proof of  Chiril\u{a}-Hamada-Kohr \cite{C4}.
\begin{thm}\label{th4.2}
Let $f\in \overline{\mathcal {S}^0_g(\mathbb{D})}$ and $F=\Phi_{n,\widehat{\alpha},\widehat{\beta}}(f).$
Assume $F\in$ {\rm supp}$\Phi_{n,\widehat{\alpha},\widehat{\beta}}(\overline{\mathcal {S}^0_g(\mathbb{D})})$. Then there exists a
 $g$-Loewner
chain $F_{\widehat{\alpha},\widehat{\beta}}(z,t)$ given by \eqref{eq4.1} with $\widehat{\alpha}\in [0, 1]$, $\widehat{\beta}\in [0, \frac{1}{2}]$,
$\widehat{\alpha}+\widehat{\beta}\leqslant1$ and $t_0>0$ such that $e^{-t}F_{\widehat{\alpha},\widehat{\beta}}(\cdot,t)\in$ {\rm supp} $\Phi_{n,\widehat{\alpha},\widehat{\beta}}(\overline{\mathcal {S}^0_g(\mathbb{D})})$
for $0\leqslant t<t_0.$
\end{thm}
\begin{proof}
Since $f\in \overline{\mathcal {S}^0_g(\mathbb{D})}$ , there exists a Loewner chain $f(z_1,t): \mathbb{D}\times[0,+\infty)\rightarrow \mathbb{C}$ such that
$f=f(\cdot,0)$, $e^{-t}f(\cdot,t)\in \overline{\mathcal {S}^0_g(\mathbb{D})}$   and $\{e^{-t}f(\cdot,t)\}$ is a normal family on $\mathbb{D}$.
Similar to the proof of Theorem \ref{th4.1}, there exist a $g$-Loewner chain $F_{\widehat{\alpha},\widehat{\beta}}(z,t)$ defined as \eqref{eq4.1}
such that
$F=F_{\widehat{\alpha},\widehat{\beta}}(\cdot,0),\,\, \{e^{-t}F_{\widehat{\alpha},\widehat{\beta}}(\cdot,t)\}$ is a normal family on $\mathbb{B}^n.$
Moreover, it is clear  $$e^{-t}F_{\widehat{\alpha},\widehat{\beta}}(\cdot,t)=\Phi_{n,\widehat{\alpha},\widehat{\beta}}(e^{-t}f(\cdot,t))\in \Phi_{n,\widehat{\alpha},\widehat{\beta}}\Big(\overline{\mathcal {S}^0_g(\mathbb{D})}\Big)$$
for $t\geqslant0$.
Let $V(z,s,t)$ be the transition mapping associated
with $F_{\widehat{\alpha},\widehat{\beta}}(z,t)$ and let $V(z,t)=V(z,0,t)$ for $z\in \mathbb{B}^n$.
\\
Since $F\in$ supp $\Phi_{n,\widehat{\alpha},\widehat{\beta}}\Big(\overline{\mathcal {S}^0_g(\mathbb{D})}\Big)$, there exist a continuous linear functional $L$ on $H(\mathbb{B}^n)$ such that
\begin{equation}\label{eq4.6}
\Re L(F)=\max\Bigg\{\Re L(M): M\in \Phi_{n,\widehat{\alpha},\widehat{\beta}}\Big(\overline{\mathcal {S}^0_g(\mathbb{D})}\Big)\Bigg\}
\end{equation}
and $\Re L$ is nonconstant on $\Phi_{n,\widehat{\alpha},\widehat{\beta}}\Big(\overline{\mathcal {S}^0_g(\mathbb{D})}\Big).$
\\
Fix $t\geqslant0.$ Let $L_t: H(\mathbb{B}^n)\rightarrow \mathbb{C}$ be the functional given by
\begin{equation}\label{eq4.7}
L_t(M)=L(e^t M\circ V(z,t)), \,\,M\in H(\mathbb{B}^n),
\end{equation}
where
\begin{equation}\label{eq4.8}
e^t M\circ V_t\in \Phi_{n,\widehat{\alpha},\widehat{\beta}}\Big(\overline{\mathcal {S}^0_g(\mathbb{D})}\Big)
\end{equation}
 for $M\in \Phi_{n,\widehat{\alpha},\widehat{\beta}}\Big(\overline{\mathcal {S}^0_g(\mathbb{D})}\Big)$ by \eqref{eq4.5}.
\\
Then $L_t$ is a continuous linear functional on $H(\mathbb{B}^n)$ and
\begin{equation}\label{eq4.9}
L_t(e^{-t}F_{\widehat{\alpha},\widehat{\beta}}(\cdot,t))=L(F_{\widehat{\alpha},\widehat{\beta}}(V(\cdot,t),t))=L(F_{\widehat{\alpha},\widehat{\beta}}(z,0))=L(F).
\end{equation}
Using \eqref{eq4.6}-\eqref{eq4.9}, then
\begin{equation}\label{eq4.10}
\Re L_t(e^{-t}F_{\widehat{\alpha},\widehat{\beta}}(\cdot,t))=\Re L(F)\geqslant \Re L(e^t M\circ V(z,t))=\Re L_t(M)
\end{equation}
for all $M\in \Phi_{n,\widehat{\alpha},\widehat{\beta}}\Big(\overline{\mathcal {S}^0_g(\mathbb{D})}\Big).$ Hence, \eqref{eq4.10} implies that
\begin{equation}\label{eq4.11}
\Re L_t(e^{-t}F_{\widehat{\alpha},\widehat{\beta}}(\cdot,t))=\max\bigg\{\Re L_t(M): M\in \Phi_{n,\widehat{\alpha},\widehat{\beta}}\Big(\overline{\mathcal {S}^0_g(\mathbb{D})}\Big)\bigg\}.
\end{equation}
 Because $F\in$ supp $\Phi_{n,\widehat{\alpha},\widehat{\beta}}\Big(\overline{\mathcal {S}^0_g(\mathbb{D})}\Big)$, there exists a $Q\in \Phi_{n,\widehat{\alpha},\widehat{\beta}}\Big(\overline{\mathcal {S}^0_g(\mathbb{D})}\Big)$
such that $\Re L(Q)<\Re L(F)$. Since $L_t(Q)\rightarrow L(Q)$ as $t\rightarrow 0^+,$ there exists $t_0>0$ such that
$$\Re L_t(Q)<\Re L(F)=\Re L_t(e^{-t}F_{\widehat{\alpha},\widehat{\beta}}(\cdot,t)),\,\, 0\leqslant t<t_0.$$
Therefore $\Re L_t|_{\Phi_{n,\widehat{\alpha},\widehat{\beta}}}\Big(\overline{\mathcal {S}^0_g(\mathbb{D})}\Big)$ is nonconstant for $0\leqslant t<t_0$.
Finally, in view of \eqref{eq4.11}, the conclusion follows, as desired. This completes the proof.
\end{proof}
\begin{rem}
(i) When $n=2,\,\, \widehat{\alpha}=0,\,\, \widehat{\beta}=\frac{1}{2}$ and $g(\xi)=\frac{1-\xi}{1+\xi},\,\, \xi\in\mathbb{D}$, Theorem \ref{th4.1} and Theorem \ref{th4.2} were proved by Graham-Kohr-Pfaltzgraff \cite{GHK5}.\\
(ii) In the case of $\widehat{\alpha}=0,\,\, \widehat{\beta}=\frac{1}{2}$ and $g(\xi)=\frac{1-\xi}{1+\xi}, \xi\in\mathbb{D}$ in Theorem \ref{th4.2}, Schleissinger \cite{SS} proved the corresponding result
for $0\leqslant t<+\infty.$ There is a problem: whether will Theorem \ref{th4.2} still hold for all $0\leqslant t<+\infty?$  The same
problem was given by Chiril\u{a} \cite{C2} when the operator $\Phi_{n,\widehat{\alpha},\widehat{\beta}}$ is replaced by the general Pfaltzgraff-Suffridge extension operator $\Psi_{n,\widehat{\alpha}}$.

\end{rem}


\subsection*{Acknowledgment}
The project is supported by the National Natural Science Foundation of China
(No. 11671306).

\end{document}